\documentclass[12pt,a4paper]{amsart}
\usepackage{amsmath,amsthm,amssymb,graphicx}

\DeclareMathOperator*{\conv}{conv}
\DeclareMathOperator*{\Argmax}{Arg\,max}

\newcommand{\N}{\mathbb{N}}
\newcommand{\R}{\mathbb{R}}

\newtheorem{claim}{Claim}
\newtheorem{conjecture}{Conjecture}

\title{The Demyanov-Ryabova conjecture is false}

\author{Vera Roshchina}
\thanks{School of Mathematics and Statistics, University of New South Wales; School of Science, RMIT University; CIAO, Federation University Australia; e-mail: v.roshchina@unsw.edu.au}

\begin{document}

\maketitle

\begin{abstract}
It was conjectured by Vladimir Demyanov and Julia Ryabova in 2011 that the minimal cycle in the sequence obtained via repeated application of the Demyanov converter to a finite family of polytopes is at most two. We construct a counterexample for which the minimal cycle has length 4.

\medskip

\end{abstract}

\section{Introduction}
The Demyanov-Ryabova conjecture is a geometric problem that comes from the calculus of exhausters. A representation of a positively homogeneous function as a lower envelope of a family of sublinear functions can be transformed into an upper envelope of a family of superlinear functions via the Demyanov converter. The converter generates superdifferentials of the latter representation from the subdifferentials of the former (we refer the reader to an extensive discussion on converters in \cite{converters} and also to the original paper \cite{exhausters} for more details on exhausters).

It was conjectured in \cite{conjecture} that the repeated application of the Demyanov converter to a finite collection of convex polytopes eventually reaches a cycle of length at most 2. The goal of this note is to describe a counterexample to this conjecture.

Let $\Omega$ be a finite collection of $n$-polytopes, i.e. $\Omega = \{P_i, i\in I\}$, where $I$ is a finite index set, and $P_i\subset \R^n$ is a convex polytope for each $i\in I$. For every $g\in \R^n$ we can construct a new polytope $P_\Omega(g)$ from the collection $\Omega$,
$$
P_{\Omega}(g):= \conv_{P\in \Omega} \Argmax_{v\in P} \langle v,g\rangle,
$$
where by $\conv$ we denote the convex hull (see \cite[Part~I \S2]{rockafellar}),  $\langle a,b\rangle = a^Tb$ is the standard inner product on $\R^n$, and  
$$
 \Argmax_{v\in P} \langle v,g\rangle = \left\{ v\in P\,|\, \langle v,g\rangle = \max_{u\in P}\langle u,g\rangle   \right\}.
$$
This construction is illustrated in Figure~\ref{fig:converter}.
\begin{figure}[ht]
{\centering\includegraphics[scale = 1]{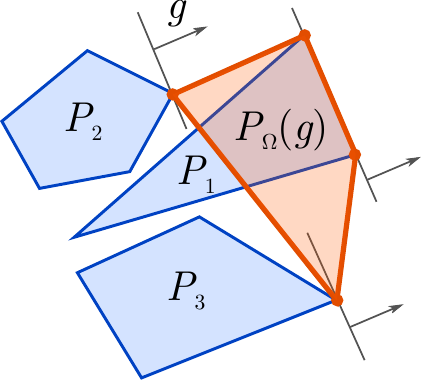}\\}
\caption{Construction of $P_\Omega(g)$ for $\Omega = \{P_1,P_2,P_3\}$.}
\label{fig:converter}
\end{figure} 
The Demyanov converter outputs the collection of all such sets for nonzero $g\in \R^n$, 
\begin{equation}\label{eq:conversion}
F(\Omega) = \{P_{\Omega}(g) \, |\, g\neq 0\}.
\end{equation}
For any finite set $\Omega$ of convex polytopes the set $F(\Omega)$ is also a finite set of convex polytopes since $F(\Omega)$ consists of convex hulls of faces of polytopes from $\Omega$, and there are finitely many of them (see \cite[Chapter~2]{ziegler} for a detailed discussion of the facial geometry of polytopes).  Hence, a repetitive application of the Demyanov converter to some initial collection $\Omega_0$ generates an infinite sequence of finite sets of polytopes $\{\Omega_k\}$,
\begin{equation}\label{eq:sequence}
\Omega_k:= F(\Omega_{k-1})\quad \forall \, k \in \N.
\end{equation}
Such sequence eventually cycles since the number of different collections of polytopes that can be constructed on the same set of vertices is finite. For any starting collection $\Omega_0$ there exist $N,l\in \N$ such that 
\begin{equation}\label{eq:cycle}
\Omega_{N+l} = \Omega_N,
\end{equation}
i.e. the sequence $\{\Omega_k\}$ has an $l$-cycle. We can choose the minimal value $L(\Omega_0) = l$ such that \eqref{eq:cycle} is true for some $N$. Vladimir Demyanov and Julia Ryabova conjectured \cite{conjecture} that this minimal length is bounded by~2. We rephrase the conjecture in our notation.

\begin{conjecture}[Demyanov and Ryabova] \label{conjecture} For any finite collection of $n$-polytopes $\Omega_0$ the minimal cycle length $L(\Omega_0)$ is at most $2$. 
\end{conjecture}

The conjecture is known to be true in some special cases. In particular, Tian Sang \cite{tian} showed that $L(\Omega_0)\leq 2$ when the extreme points of the polytopes in $\Omega_0$ are affinely independent, while Aris Daniilidis and Colin Petitjean \cite{aris} confirmed the conjecture under the condition that $\Omega_0$ contains certain sets and the vertices of polytopes in $\Omega_0$ are the vertices of $\conv \Omega_0$. We disprove the conjecture by providing an explicit counterexample in $\R^2$. We note that our counterexample does not satisfy the conditions of neither \cite{aris} nor \cite{tian}.

\section{Counterexample}

Our counterexample $\Omega_0$ is the collection of four polytopes in $\R^2$ shown in Figure~\ref{fig:Omega-naught}.
\begin{figure}[ht]
{\centering\includegraphics[scale = 1]{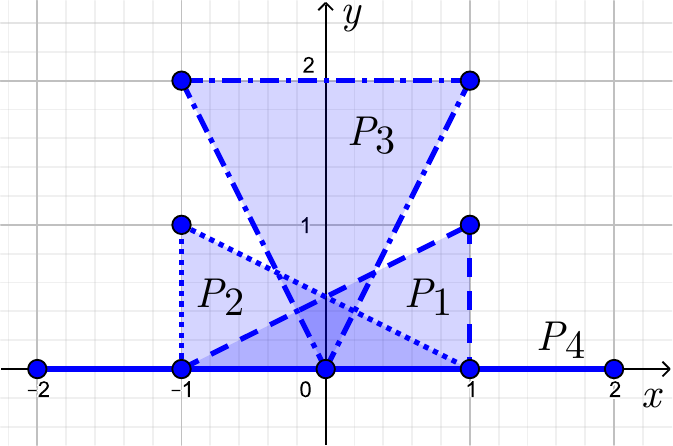}\\}
\caption{The counterexample $\Omega_0 = \{P_1,P_2,P_3,P_4\}$.}
\label{fig:Omega-naught}
\end{figure}
In the next claim we provide an explicit algebraic description of $\Omega_0$ and demonstrate that $L(\Omega_0) = 4$.

\begin{claim} Let $\Omega_0 = \{P_1,P_2,P_3,P_4\}$ be the collection of two-dimensional convex polygons,
\begin{align}\label{eq:OmegaC}
P_1 & =  \conv\left\{(1,0),(1,1),(-1,0)\right\},\notag \\
P_2 & =  \conv\left\{(-1,0),(-1,1),(1,0)\right\},\notag\\ 
P_3 & =  \conv\left\{(1,2),(-1,2),(0,0)\right\}, \\ 
P_4 & =  \conv\left\{(2,0),(-2,0)\right\},\notag
\end{align}
and let $\Omega_1, \dots \Omega_5$ be the first five members of  the sequence generated from $\Omega_0$ by the conversion operator \eqref{eq:conversion} as in \eqref{eq:sequence}. Then 
\begin{equation}\label{eq:contradiction}
\Omega_5 = \Omega_1, \text{ but } \;\Omega_5\neq \Omega_3.
\end{equation}
\end{claim}
\begin{proof} The claim can be proved by constructing $\Omega_1$, \dots, $\Omega_5$ explicitly and demonstrating that the relations \eqref{eq:contradiction} are true.

Since the family $\Omega_0$ is symmetric with respect to the vertical axis, it is sufficient to compute $P_{\Omega_0}(g)$ for the right half-plane only $\{g\,|\, g_x\geq 0, g\neq 0\}$. The remaining sets can be obtained by reflecting these polytopes through the $y$-axis.

We have  
$$
\Argmax_{v\in P_1} \langle v,g\rangle =
\begin{cases}
\conv\{(-1,0),(1,0)\}, & g_x = 0, g_y<0,\\
\{(1,0)\}, & g_x>0,g_y<0,\\
\conv \{(1,0),(1,1)\}, & g_x > 0, g_y=0,\\
\{(1,1)\}, & g_y > 0, g_x \geq 0.
\end{cases} 
$$
$$
\Argmax_{v\in P_2} \langle v,g\rangle =
\begin{cases}
\conv \{(-1,0),(1,0)\}, & g_x = 0, g_y<0,\\
\{(1,0)\}, & g_x>0,g_y< 2 g_x,\\
\conv \{(1,0),(-1,1)\}, & g_x > 0, g_y=2 g_x,\\
\{(-1,1)\}, & g_x \geq 0, g_y > 2g_x.
\end{cases} 
$$
$$
\Argmax_{v\in P_3} \langle v,g\rangle =
\begin{cases}
\{(0,0)\}, & g_x \geq  0, g_y<-\frac{1}{2} g_x,\\
\conv \{(0,0),(1,2)\}, & g_x> 0,g_y  = -\frac{1}{2} g_x,\\
\{(1,2)\}, & g_x > 0, g_y > -\frac{1}{2}g_x\\
\conv \{(-1,2),(1,2)\}, & g_y > 0, g_x = 0.
\end{cases} 
$$
$$
\Argmax_{v\in P_4} \langle v,g\rangle =
\begin{cases}
\conv \{(-2,0), (2,0)\}, & g_x =   0, g_y<0,\\
\{(2,0)\}, & g_x> 0\\
\conv \{(-2,0), (2,0)\}, & g_x =   0, g_y>0.
\end{cases} 
$$
Aligning the cases for each of the computed faces, we obtain
$$
P_{\Omega_0}(g) = \begin{cases}
\conv \{(-2,0),(2,0)\} & g_x = 0, g_y<0,\\
\conv \{ (0,0),(2,0)\}& g_x > 0, g_y< -\frac{1}{2}g_x,\\
\conv \{(0,0),(1,2),(2,0) \}& g_x > 0, g_y= -\frac{1}{2}g_x,\\
\conv \{(1,0),(1,2),(2,0) \} & g_x > 0,  -\frac{1}{2}g_x<g_y< 2 g_x,\\
\conv \{(1,0),(-1,1),(1,2),(2,0)\} & g_x>0,g_y= 2 g_x,\\
\conv \{(-1,1),(1,2),(2,0)\} & g_x>0,2 g_x<g_y,\\
\conv \{(-1,2),(1,2),(2,0),(-2,0)\}  & g_x = 0, 0<g_y.
\end{cases}
$$
Reflecting these sets via the $y$-axis we obtain the full collection of sets in $\Omega_1$ shown in Figure~\ref{fig:Omega-one}.
\begin{figure}[ht]
{\centering\includegraphics[width=0.9\textwidth]{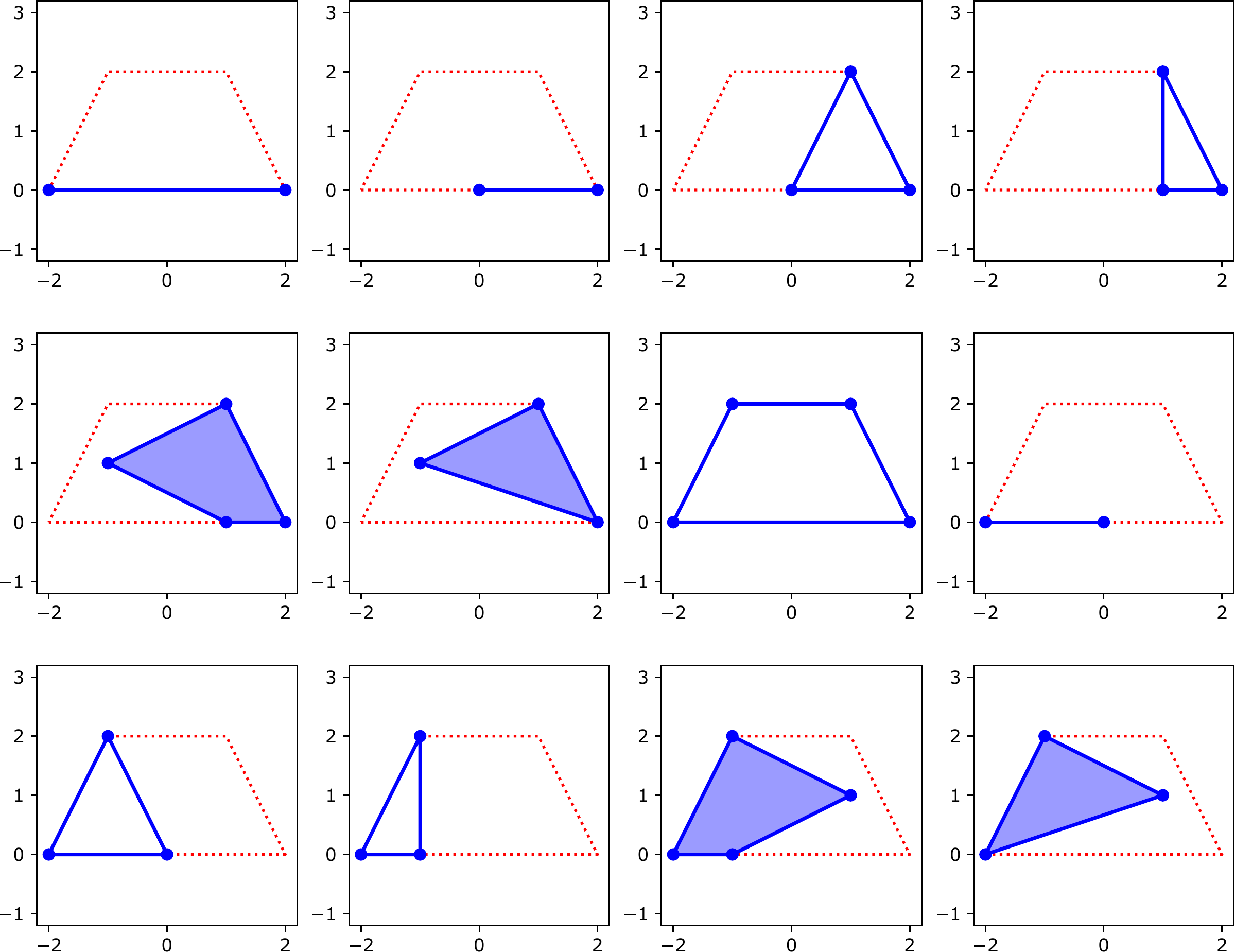}\\}
\caption{The set $\Omega_1=\Omega_5$.}
\label{fig:Omega-one}
\end{figure}
It is a technical exercise to perform the conversion four more times obtaining the collections shown in Figs~\ref{fig:Omega-two}-\ref{fig:Omega-four}, as well as the collection $\Omega_5$ that coincides with $\Omega_1$. Explicitly, we have the following sets for $g_x\geq 0$ (recall that the configuration is symmetric, so the remaining sets can be obtained as the reflections through the $y$-axis). In our notation $P_{\Omega_i}(g)$ corresponds to the sets in $\Omega_{i+1}$.
\begin{figure}[ht]
{\centering\includegraphics[width=0.9\textwidth]{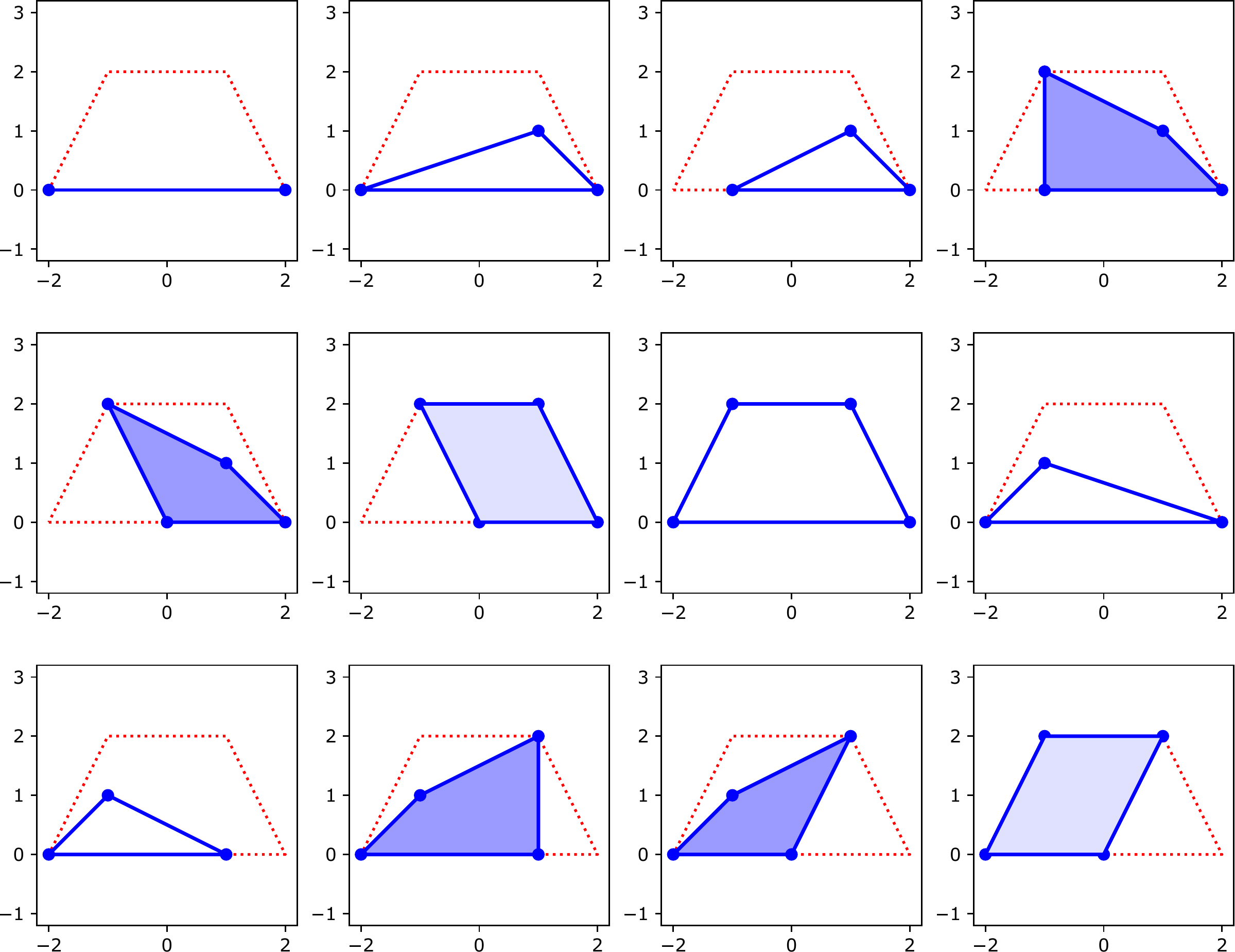}\\}
\caption{The set $\Omega_2$.}
\label{fig:Omega-two}
\end{figure}
\begin{figure}[ht]
{\centering\includegraphics[width=0.9\textwidth]{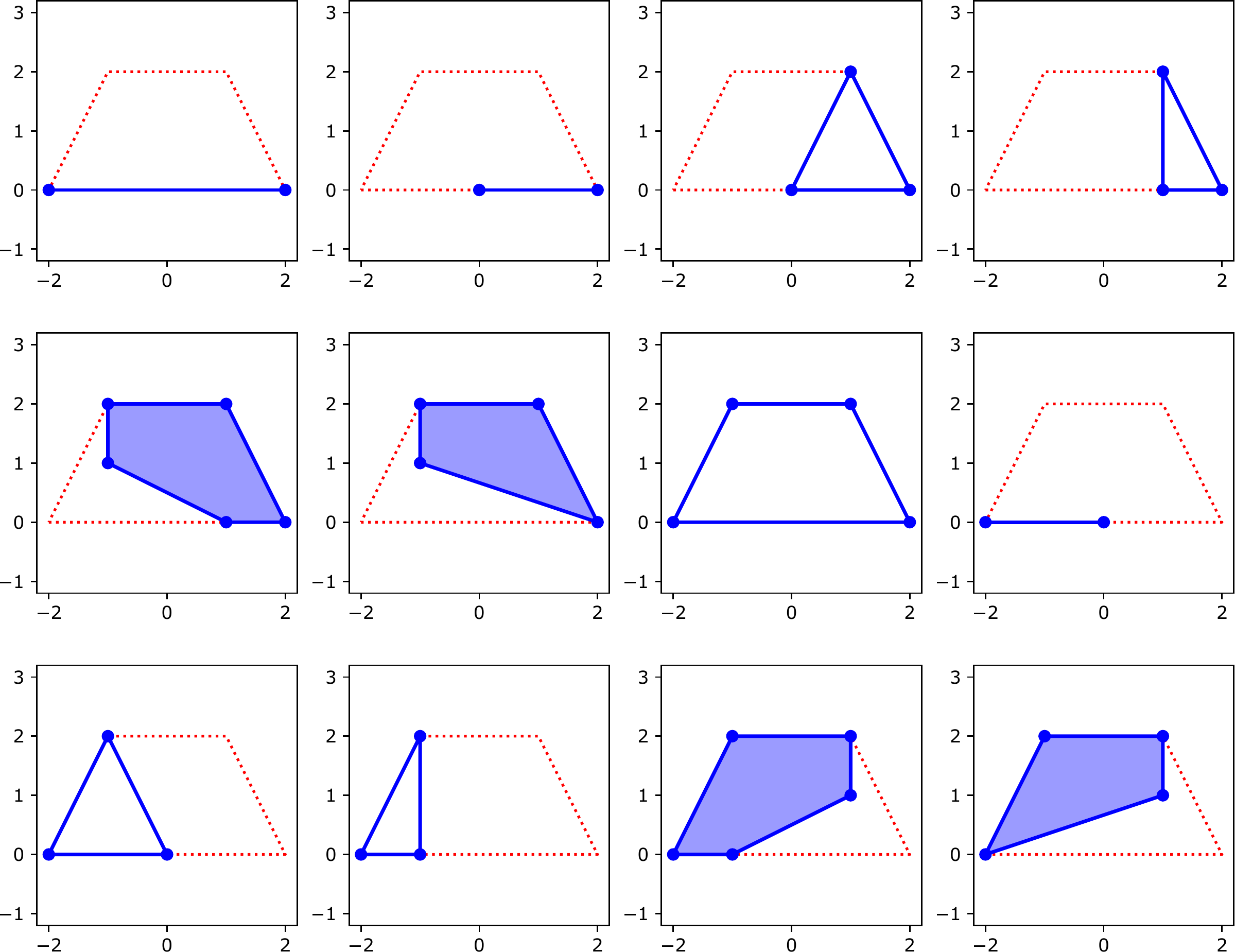}\\}
\caption{The set $\Omega_3$.}
\label{fig:Omega-three}
\end{figure}
\begin{figure}[ht]
{\centering\includegraphics[width=0.9\textwidth]{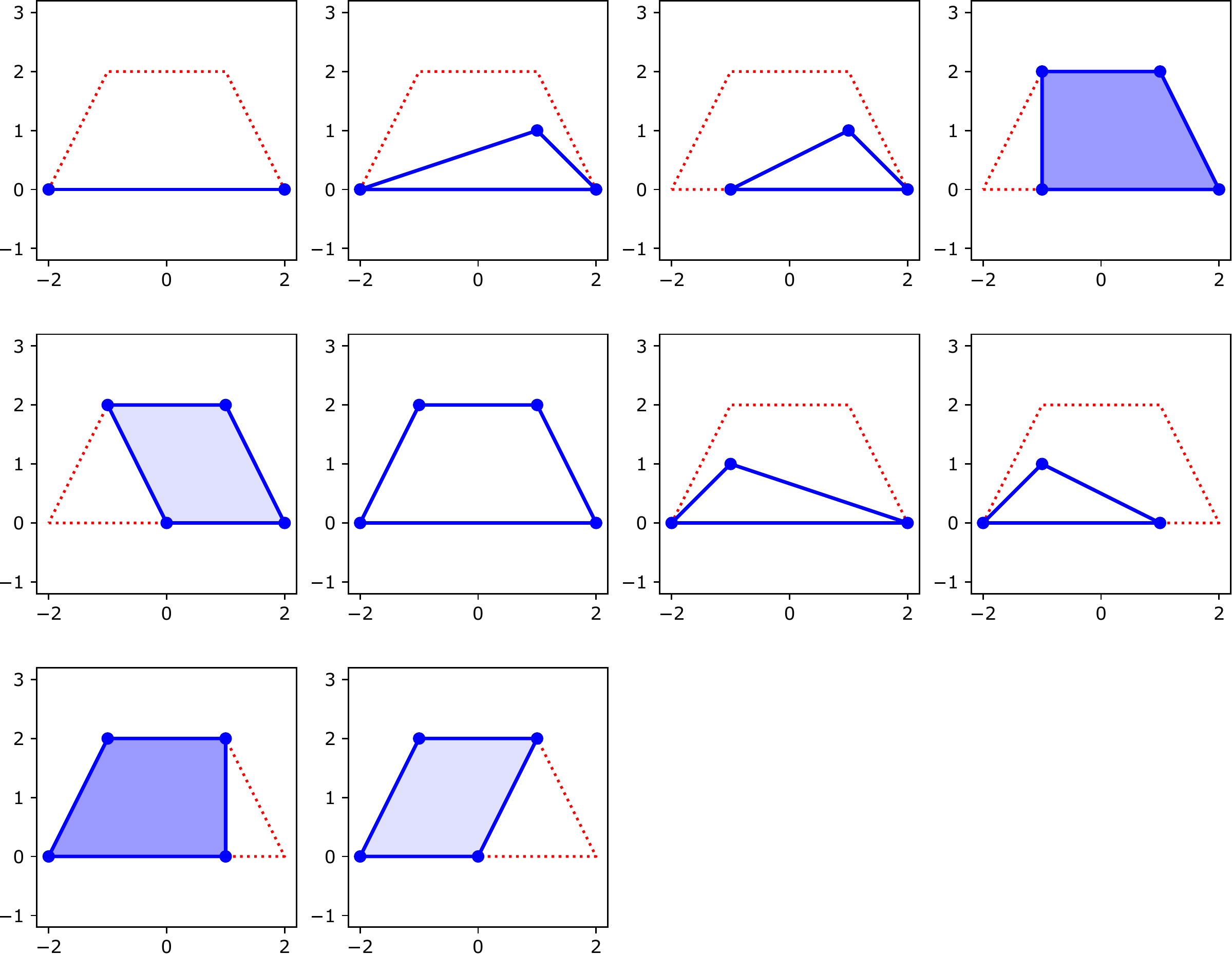}\\}
\caption{The set $\Omega_4$.}
\label{fig:Omega-four}
\end{figure}
$$
P_{\Omega_1}(g) = \begin{cases}
\conv \{(-2,0),(2,0)\} & g_x\geq 0, g_y<-3 g_x,\\
\conv \{ (-2,0),(2,0),(1,1)\}& g_x > 0, g_y = -3 g_x,\\
\conv \{(-1,0),(1,1),(2,0) \}& g_x > 0, -3 g_x<g_y<0,\\
\conv \{(-1,0),(-1,2),(2,0),(1,1) \} & g_x > 0,  g_y=0,\\
\conv \{(0,0),(-1,2),(2,0),(1,1)\} & g_x>0,0<g_y<\frac{1}{2} g_x,\\
\conv \{(0,0),(-1,2),(2,0),(1,2)\} & g_x>0,\frac{1}{2} g_x\leq g_y,\\
\conv \{(-1,2),(1,2),(2,0),(-2,0)\}  & g_x = 0, 0<g_y.
\end{cases}
$$
$$
P_{\Omega_2}(g) = \begin{cases}
\conv \{(-2,0),(2,0)\} & g_x = 0, g_y<0,\\
\conv \{ (0,0),(2,0)\}& g_x > 0, g_y< -\frac{1}{2}g_x,\\
\conv \{(0,0),(1,2),(2,0) \}& g_x > 0, g_y= -\frac{1}{2}g_x,\\
\conv \{(1,0),(1,2),(2,0) \} & g_x > 0,  -\frac{1}{2}g_x<g_y< 2 g_x,\\
\conv \{(1,0),(-1,1),(1,2),(2,0),(-1,2)\} & g_x>0,g_y= 2 g_x,\\
\conv \{(-1,1),(1,2),(2,0),(-1,2)\} & g_x>0,2 g_x<g_y,\\
\conv \{(-1,2),(1,2),(2,0),(-2,0)\}  & g_x = 0, 0<g_y.
\end{cases}
$$
$$
P_{\Omega_3}(g) = \begin{cases}
\conv \{(-2,0),(2,0)\} & g_x = 0, g_y< -3 g_x,\\
\conv \{(-2,0),(2,0),(1,1)\} & g_x > 0, g_y= -3 g_x,\\
\conv \{(-1,0),(1,1),(2,0) \}& g_x > 0, -3 g_x<g_y<0,\\
\conv \{(-1,0),(-1,2),(2,0),(1,2) \} & g_x > 0,  g_y=0,\\
\conv \{(0,0),(-1,2),(2,0),(1,2)\} & g_x>0,g_y>0,\\
\conv \{(-1,2),(1,2),(2,0),(-2,0)\}  & g_x = 0, g_y>0.
\end{cases}
$$
The sets $\Omega_5$ and $\Omega_1$ coincide, while $\Omega_1$ and $\Omega_3$ are different. In particular, for $\{g\,|\,g_x>0, g_y\geq 2 g_x\}$ the sets $\Omega_3$ contain an additional point $(-1,2)$ as compared to $\Omega_1$ and $\Omega_5$ (see the shaded sets in Figs.~\ref{fig:Omega-one} and~\ref{fig:Omega-three}). We have hence shown \eqref{eq:contradiction}.
\end{proof}

\section{Acknowledgements}

The author is grateful to the Australian Research Council for continuing financial support via projects  DE150100240 and DP180100602, also to the MATRIX research institute for organising the recent program in algebraic geometry, approximation and optimisation, which provided a fertile research environment that helped this discovery.

\bibliographystyle{plain}
\bibliography{refs}

\end{document}